\documentclass[12pt]{article}
\usepackage{amsmath}
\usepackage{amssymb}
\usepackage{latexsym}
\usepackage{amsthm}
\usepackage{mathrsfs}
\usepackage{enumitem}
\usepackage[colorlinks,
            linkcolor=blue,
            anchorcolor=blue,
            citecolor=green
            ]{hyperref}
\usepackage{graphicx}
\usepackage{diagbox}
\usepackage{tikz}

\newtheorem{thm}{Theorem}[section]
\newtheorem{lem}[thm]{Lemma}
\newtheorem{prop}[thm]{Proposition}
\newtheorem{cor}[thm]{Corollary}

\newtheorem{definition}{Definition}
\newtheorem{prob}{Problem}

\newcommand*\pFqskip{8mu}
\catcode`,\active
\newcommand*\pFq{\begingroup
        \catcode`\,\active
        \def ,{\mskip\pFqskip\relax}%
        \dopFq
}
\catcode`\,12
\def\dopFq#1#2#3#4#5{%
        {}_{#1}F_{#2}\biggl[\genfrac..{0pt}{}{#3}{#4};#5\biggr]%
        \endgroup
}

\makeatother

\begin{document}

\begin{center}
{\large \bf  Avoiding vincular patterns on alternating words}
\end{center}

\begin{center}
Alice L.L. Gao$^{1}$,
Sergey Kitaev$^{2}$, and Philip B. Zhang$^{3}$\\[6pt]

$^{1}$Center for Combinatorics, LPMC-TJKLC \\
Nankai University, Tianjin 300071, P. R. China\\[6pt]

$^{2}$ Department of Computer and Information Sciences \\
University of Strathclyde, 26 Richmond Street, Glasgow G1 1XH, UK\\[6pt]

$^{3}$ College of Mathematical Science \\
Tianjin Normal University, Tianjin  300387, P. R. China\\[6pt]

Email: $^{1}${\tt gaolulublue@mail.nankai.edu.cn},
	   $^{2}${\tt sergey.kitaev@cis.strath.ac.uk},
           $^{3}${\tt zhangbiaonk@163.com}
\end{center}

\noindent\textbf{Abstract.}
A word $w=w_1w_2\cdots w_n$ is alternating if either $w_1<w_2>w_3<w_4>\cdots$ (when the word is up-down) or $w_1>w_2<w_3>w_4<\cdots$ (when the word is down-up). The study of alternating words avoiding classical permutation patterns was initiated by the authors in~\cite{GKZ}, where, in particular, it was shown that 123-avoiding up-down words of even length are counted by the Narayana numbers.

However, not much was understood on the structure of 123-avoiding up-down words. In this paper, we fill in this gap by introducing the notion of a cut-pair that allows us to subdivide the set of words in question into equivalence classes. We provide a combinatorial argument to show that the number of equivalence classes is given by the Catalan numbers, which induces an alternative (combinatorial) proof of the corresponding result in~\cite{GKZ}.

Further, we extend the enumerative results in~\cite{GKZ} to the case of alternating words avoiding a vincular pattern of length 3. We show that it is sufficient to enumerate up-down words of even length avoiding the consecutive pattern $\underline{132}$ and up-down words of odd length avoiding the consecutive pattern $\underline{312}$ to answer all of our enumerative questions. The former of the two key cases is enumerated by the Stirling numbers of the second kind.



\section{Introduction}\label{intro}
A permutation $\pi=\pi_1\pi_2\cdots \pi_n$ is called {\em up-down} if $\pi_1<\pi_2>\pi_3<\pi_4>\pi_5<\cdots$.  A permutation $\pi=\pi_1\pi_2\cdots \pi_n$ is called {\em down-up} if $\pi_1>\pi_2<\pi_3>\pi_4<\pi_5>\cdots$. A famous result of Andr\'{e} states that if $E_n$ is the number of up-down (equivalently, down-up) permutations of $1,2,\ldots,n$, then $$\sum_{n\geq 0}E_n\frac{x^n}{n!}=\sec x+\tan x.$$ Some aspects of up-down and down-up permutations\footnote{Up-down and down-up permutations are also called in the literature reverse alternating and alternating permutations, respectively.} are surveyed in~\cite{Stanley2010survey}. Slightly abusing these definitions, we refer to {\em alternating permutations} as the union of up-down and down-up permutations\footnote{The union of up-down and down-up permutations is also known in the literature as the set of zigzag permutations.}. The study of alternating permutations was extended to other types of alternating sequences, for example, to up-down multi-permutations~\cite{Carlitz}. For other relevant sources see~\cite{CarlitzScoville} and~\cite{Carlitz1}.

In \cite{GKZ} we extended the study of alternating permutations to that of {\em alternating words}. These words, also called {\em zigzag words}, are the union of up-down and down-up words, which are defined in a similar way to the definition of up-down and down-up permutations, respectively. Namely, a word $w=w_1w_2\cdots w_n$ is up-down (resp., down-up) if $w_1<w_2>w_3<\cdots$ (resp., $w_1>w_2<w_3>\cdots$)\footnote{We note that there are other ways to extend the notion of alternating permutations to words. For example, one can replace ``$>$'' and ``$<$'' by ``$\geq$'' and ``$\leq$'', respectively, in the definition of alternating words to define what we call weak alternating words. Weak alternating words are not in the scope of this paper.}. For example, $1214$, $2413$, $2424$ and $3434$ are examples of up-down words of length 4 over the alphabet $\{1,2,3,4\}$. In this paper, we write the entries of an up-down word $w$ as $w=b_1 t_1 b_2 t_2 \cdots $ where $b_i<t_i>b_{i+1}$ for $i\geq 1$. We call a letter $b_i$ a {\em bottom element} and $t_i$ a {\em top element}.

For a word $w=w_1w_2\cdots w_n$ over the alphabet $\{1,2,\ldots,k\}$, its complement $w^c$ is the word $c_1c_2\cdots c_n$, where for each $i=1,2,\ldots,n$, $c_i=k+1-w_i$. For example, the complement of the word $24265$ over the alphabet $\{1,2,\ldots,6\}$ is $53512$. For a word $w=w_1w_2\cdots w_n$, its {\em reverse} $w^r$ is the word $w_nw_{n-1}\cdots w_1$. For example, if $w=53512$ then $w^r=21535$.

We say that a permutation $\pi=\pi_1\pi_2\cdots\pi_n$ {\em contains an occurrence} of a {\em pattern} $\tau=\tau_1\tau_2\cdots\tau_k$ if there are $1\leq i_1< i_2<\cdots< i_k\leq n$ such that $\pi_{i_1}\pi_{i_2}\cdots \pi_{i_k}$ is order-isomorphic to $\tau$. If $\pi$ does not contain an occurrence of $\tau$, we say that $\pi$ {\em avoids}~$\tau$. For example, the permutation 315267 contains several occurrences of the pattern 123, such as, the subsequences 356 and 157, while this permutation avoids the pattern 321. Such patterns are referred to as ``classical patterns'' in the theory of patterns in permutations and words (see \cite{Kitaev2011Patterns} for a comprehensive introduction to the theory). Occurrences of a pattern in words are defined similarly as subsequences order-isomorphic to a given word called pattern. The only difference between word and permutation patterns is that word patterns can contain repetitive letters, which is not in the scope of this paper.

Another type of patterns of interest to us is {\em vincular patterns}, also known as {\em generalized patterns} \cite{BS}, in occurrences of which some of the letters may be required to be adjacent in a permutation or a word. We underline letters of a given pattern to indicate the letters that must be adjacent in any occurrence of the pattern. For example, the word $w=1244254$ contains four occurrences of the pattern $1\underline{32}$, namely, the subsequences 142, 154, and 254 twice: in each of these occurrences, the letters in $w$ corresponding to 2 and 3 in the pattern stay next to each other. On the other hand, $w$ contains just one occurrence of the pattern  $\underline{132}$ formed by the rightmost three letters in $w$. If all letters in an occurrence of a pattern are required to stay next to each other, which is indicated by underlying all letters in the pattern, such patterns are called {\em consecutive patterns}. Vincular patterns play an important role in the theory of patterns in permutations and words (see Sections 3.3 and 3.4 in \cite{Kitaev2011Patterns} for details).

In this paper, $[k]=\{1,2,\ldots,k\}$, $S^{p}_{k,n}$ denotes the set of  $p$-avoiding up-down words of length $n$ over $[k]$, and $N^{p}_{k,n}$ denotes the number of words in $S^{p}_{k,n}$.  Two patterns, $p_1$ and $p_2$, are {\em Wilf-equivalent} if $N^{p_1}_{k,n}=N^{p_2}_{k,n}$ for $n\geq 0$ and $k\geq 1$. Also, for a word $w$, $\{w\}^+$ denotes a word in $\{w,ww,www,\ldots\}$ and $\{w\}^*$ denotes a word in $\{w\}^+\cup \{\epsilon\}$, where $\epsilon$ is the empty word. Moreover, recall that the $n$-th {\em Catalan number} is $C_n=\frac{1}{n+1}{2n \choose n}$ and the {\em Narayana number} $N_{n,m}$ is $\frac{1}{m+1}{n\choose m}{n-1\choose m}$. Also, a Dyck path of semi-length $n$ is a lattice path with steps $(1,1)$ and $(1,-1)$ which begins at $(0,0)$, ends at $(2n,0)$, and never goes below the $x$-axis.

The content of this paper is as follows. In Section~\ref{struct-sec} we not only discuss in more detail the structure of 123-avoiding up-down words of even length, but also give an alternative, combinatorial way to show that the number of these words is given by the Narayana numbers. Originally, this fact was established in~\cite{GKZ}. An essential part of our studies here is the notion of a  {\em cut-pair}, which allows us to subdivide the set of words in question into equivalence classes. We prove that the number of equivalence classes is counted by the {\em Catalan numbers}, which is done by establishing a bijection between the classes and {\em Dyck paths} of certain length.

Further, in Sections~\ref{three-consecutive-up-down-sec} and \ref{enum-vinc-pat} we extend the enumerative results in~\cite{GKZ} to the case of alternating words avoiding a vincular pattern of length 3. This direction of research is also an extension of vincular pattern-avoidance results on all words to alternating words; see \cite[Section 7.2]{Kitaev2011Patterns} for a survey of the respective results.

\begin{table}[htbp]
\begin{center}
\begin{tabular}{|c|cccccc|cccccc|}
\hline
&$\underline{123}$&$\underline{132}$&$\underline{213}$
&$\underline{231}$&$\underline{312}$&$\underline{321}$\\
\hline
even length & $K$&$A$&$A$&$C$&$C$&$K$\\
\hline
odd length & $L$&$B$&$D$&$B$&$D$&$L$\\
\hline
\end{tabular}
\end{center}

\begin{center}
\begin{tabular}{|c|cccccc|cccccc|}
\hline
&$1\underline{23}$&$1\underline{32}$&$2\underline{13}$
&$2\underline{31}$&$3\underline{12}$&$3\underline{21}$\\
\hline
even length & $A$&$A$&$N$&$N$&$C$&$E$\\
\hline
odd length & $F$&$B$&$H$&$G$&$D$&$D$\\
\hline
\end{tabular}
\end{center}

\begin{center}
\begin{tabular}{|c|cccccc|c|}
\hline
&$\underline{12}3$&$\underline{13}2$&$\underline{21}3$
&$\underline{23}1$&$\underline{31}2$&$\underline{32}1$\\
\hline
even length &$A$&$N$&$A$&$C$&$N$&$E$\\
\hline
odd length &$D$&$G$&$D$&$B$&$H$&$F$\\
\hline
\end{tabular}
\end{center}
\caption{Wilf-equivalence for the enumerative results in this paper. The results encoded by A--L are given by Theorems~\ref{trivial-case}, \ref{thm 1.4},  \ref{thm 1.6}, \ref{thm 1.10}, \ref{thm 1.11}, \ref{thm 2.6}, \ref{thm 2.8}, \ref{thm 2.9} and Corollary \ref{cor 2.5}.}
\label{overview-enum-results}
\end{table}

Table~\ref{overview-enum-results} shows Wilf-equivalent classes, where $A$ is given by Theorem~\ref{thm 1.4}, $B$ by Theorem~\ref{thm 1.6}, $C$ by Theorem~\ref{thm 1.11},  and $D$ by Theorem~\ref{thm 1.10}. Also, $G,$ $H,$ $N$ are given by Corollary~\ref{cor 2.5} and Theorem~\ref{thm 2.6}, and  $E$ and $F$ by Theorems~\ref{thm 2.8} and~\ref{thm 2.9}. Finally, we do not give separate enumeration for $K$ and $L$, but treat these cases together in Theorem~\ref{trivial-case} by providing a recurrence relation for these numbers. In particular, we show that it is sufficient to enumerate up-down words of even length avoiding the consecutive pattern \underline{132} (corresponding to $A$ in Table~\ref{overview-enum-results}) and up-down words of odd length avoiding the consecutive pattern \underline{312} (corresponding to $D$ in Table~\ref{overview-enum-results}) to deduce all of our enumerative results. Note that $A$ in Table~\ref{overview-enum-results} is given by the {\em Stirling numbers of the second kind} $S(n, m)$ counting the number of ways to partition a set of $n$ elements into $m$ nonempty subsets.

All our results in this paper are for up-down pattern-avoiding words. However, they can be easily turned into results on down-up pattern-avoiding words by using the complement operation.

\section{Structure of 123-avoiding up-down words of even length}\label{struct-sec}

Recall that 123-avoiding up-down words were enumerated in~\cite{GKZ}.   To be more precise, the following theorem was proved in~\cite{GKZ}.

\begin{thm}[\cite{GKZ}]\label{thm-from-prev-paper}
For $p\in \{123, 132, 312, 213, 231 \}$ and $i\geq 1$, $$N^{p}_{k,2i}=N_{k+i-1,i},$$ where $N_{k,j}$, for $0\leq j\leq k-1$, is the Narayana number $\frac{1}{j+1}\binom{k}{j}\binom{k-1}{j}$.
\end{thm}

In this section, we give more details on the structure of 123-avoiding up-down words, and provide an alternative, combinatorial proof for their enumeration.

\subsection{Cut-pairs and cut-equivalence}

We begin with a description of the structure of 123-avoiding up-down words of even length.

\begin{lem}\label{lem-decreasing}
An up-down word $w=b_1 t_1 b_2 t_2 \cdots b_i t_i$ is $123$-avoiding if and only if
the following two conditions hold:
\begin{itemize}
\item[(a)] $b_1\geq b_2\geq\dots \geq b_i$,
\item[(b)] $t_1\geq t_2\geq\dots \geq t_i$.
\end{itemize}
\end{lem}

\begin{proof}
We first  show that if $w$ is a $123$-avoiding up-down word, then (a) and (b) hold.
(a) is true since if there exist $1\leq j_1<j_2\leq i$ such that $b_{j_1} < b_{j_2}$, then $b_{j_1}b_{j_2}t_{j_2}$ forms the pattern $123$.
Similarly, (b) is true since if there exist $1\leq j_1<j_2\leq i$ such that $t_{j_1} < t_{j_2}$, then $b_{j_1}t_{j_1}t_{j_2}$ forms the pattern $123$.

We next  prove that any up-down word $w$ satisfying (a) and (b) must be $123$-avoiding.
Suppose that  there is an occurrence $xyz$ of the pattern $123$ in $w$.
Then  at most one of the three letters $x$, $y$ and $z$ can stay in bottom positions, since otherwise it would  contradict  the condition (a).
Similarly, due to (b), at most one of the three letters  can stay in top positions.
This is impossible and thus $w$ is  $123$-avoiding, which completes the proof.
\end{proof}

\begin{definition}
Given a word $w=b_1 t_1 b_2 t_2 \cdots b_i t_i\in S^{123}_{k,2i}$, suppose that all pairs $b_jt_j$ ($1\leq j\leq i$) in $w$ are {\bf distinct}. Then $b_jt_j$ is a \emph{cut-pair} if
\begin{itemize}
 \item  $1 < b_j < k-1$ and $b_j>b_m$ for all $j+1\leq m\leq i$,
and
 \item  $2 < t_j < k$ and $t_j<t_m$ for $1\leq m\leq j-1$.
\end{itemize}
Furthermore, if $w$ {\bf contains repeated pairs} then $b_jt_j$ is a cut-pair if it is a cut-pair in the word obtained from $w$ by removing {\bf all} repetitions of repeated pairs.
\end{definition}

For example, given $k=6$, the word $w=4645252512 \in S^{123}_{6,10}$ has only one cut-pair $45$. On the other hand, 25 is a cut-pair in the word $252525\in S^{123}_{6,10}$. For yet another example, the word $262626\in S^{123}_{6,10}$ has no cut-pair. The word ``cut'' in ``cut-pair'' came in analogy with the notion of a {\em cut-point} in a permutation that can be used to define {\em reducible}/{\em irreducible permutations} \cite{Kitaev2011Patterns}. A cut-point in that context is a place in the permutation, where every element to the left of the place is smaller than every element to the right of it.

Combining the definition of cut-pairs with Lemma~\ref{lem-decreasing}, it is easy to see that if a $123$-avoiding up-down word $w=b_1 t_1 b_2 t_2 \cdots b_i t_i$ has cut-pairs
$b_{p_1}  t_{p_1}, b_{p_2}  t_{p_2}, \ldots,  b_{p_j}  t_{p_j}$, then we must have $k-1>b_{p_1} > b_{p_2} >\cdots > b_{p_j}>1$ and
$k>t_{p_1} > t_{p_2} >\cdots > t_{p_j}>2$.

\begin{definition}
Two words $w_1,w_2\in S^{123}_{k,2i}$ are \emph{cut-equivalent} if their sets of cut-pairs are the same.
\end{definition}

Clearly, ``to be cut-equivalent'' is an equivalence relation on  $S^{123}_{k,2i}$, and the corresponding equivalence classes are  uniquely characterized by  the cut-pairs. Let $\mathcal{F}^{123}_{k,2i}$ denote the set of cut-equivalence classes of $\mathcal{F}^{123}_{k,2i}$. For any cut-equivalence class $f$ in $\mathcal{F}^{123}_{k,2i}$, denote by $n(f)$ the number of cut-pairs each word in $f$ has (this number is the same for any word in $f$ by definition).

\begin{lem}\label{lem-cutpair}
The cut-equivalence class in $\mathcal{F}^{123}_{k,2i}$
with cut-pairs  $b_{p_1}  t_{p_1}, b_{p_2}  t_{p_2}, \ldots,  b_{p_j}  t_{p_j}$, where $p_1<p_2<\cdots< p_j$, consists of the words of length $2i$ that can be generated from the expression
\begin{align}
 & \{(k-1)k\}^{*} \{(k-2)k\}^{*} \cdots  \{b_{p_1} k\}^{*} \{b_{p_1}(k-1)\}^{*} \cdots  \{b_{p_1} t_{p_1}\}^{+} \{(b_{p_1}-1)  t_{p_1}\}^{*} \cdots  \notag \\[5pt]
  & \{b_{p_2} t_{p_1}\}^{*} \{b_{p_2} (t_{p_1}-1)\}^{*}    \cdots  \{b_{p_j}  t_{p_j}\}^{+} \cdots  \{1 t_{p_j}\}^{*} \cdots \{12\}^{*},
 \label{form_2}
\end{align}
where the second line is continuation of the first one.  Moreover, two different expressions of the form {\em (\ref{form_2})} cannot generate the same word.
\end{lem}

\begin{proof}
All the words of length $2i$ that can be generated from the expression \eqref{form_2} are clearly 123-avoiding. Moreover, by the definition of cut-pairs, these words have cut-pairs $b_{p_1}  t_{p_1}, b_{p_2}  t_{p_2}, \ldots,  b_{p_j}  t_{p_j}$.

Conversely, we shall show that any 123-avoiding word $w$ with cut-pairs  $b_{p_1}  t_{p_1}, b_{p_2}  t_{p_2}, \ldots,$  $b_{p_j}  t_{p_j}$ can be generated from \eqref{form_2}. Without loss of generality, suppose that all the pairs $b_jt_j$ in $w$ are distinct. We first prove the following claim.

\noindent \textbf{Claim:} If there is a pair $b_xt_x$ between $b_{p_m}  t_{p_m}$ and $b_{p_{m+1}}  t_{p_{m+1}}$ ($1\leq m\leq j-1$), then $b_{p_m}>  b_x>b_{p_{m+1}}$ and $  t_{p_m}>t_x >t_{p_{m+1}}$
cannot be satisfied at the same time.\\[0.1pt]

\noindent \textit{Proof of the Claim.} If some pair $b_xt_x$ between $b_{p_m}  t_{p_m}$ and $b_{p_{m+1}}  t_{p_{m+1}}$ has the property $b_{p_m}>  b_x>b_{p_{m+1}}$ and $  t_{p_m}>t_x >t_{p_{m+1}}$, then there are two cases to consider.

\noindent {\bf Case 1:} There exists at least one pair $b_yt_y$ between $b_xt_x$ and
$b_{p_{m+1}}  t_{p_{m+1}}$ in $w$ such that $b_y=b_x$. We list all such pairs as $b_{y_1} t_{y_1},b_{y_2} t_{y_2},\ldots,b_{y_s} t_{y_s}$,
where $b_x=b_{y_1}=b_{y_2}=\cdots=b_{y_s}$ and $t_x>t_{y_1}> t_{y_2}>\cdots>t_{y_s}$. But then $b_{y_s} t_{y_s}$ must be a cut-pair by the definition, which contradicts the assumption that there is no cut-pair between $b_{p_m}  t_{p_m}$ and $b_{p_{m+1}}  t_{p_{m+1}}$.

\noindent {\bf Case 2:} There exist no pair $b_yt_y$ between $b_xt_x$ and
$b_{p_{m+1}}  t_{p_{m+1}}$ such that $b_y=b_x$. Then there must exist at least one pair $b_zt_z$ such that $t_z=t_x$ between $b_{p_m}  t_{p_m}$ and
$b_x  t_x$ in $w$, since otherwise $b_xt_x$ would be a cut-pair.
We list all such pairs as $b_{z_1} t_{z_1},b_{z_2} t_{z_2},\ldots,b_{z_s} t_{z_s}$, where $b_{z_1}>b_{z_2}>\cdots>b_{z_s}>b_x$ and $t_{z_1}= t_{z_2}=\cdots=t_{z_s}=t_x$. But then $b_{z_1}t_{z_1}$ must be a cut-pair, which also contradicts the assumption that there is no cut-pair between $b_{p_m}  t_{p_m}$ and $b_{p_{m+1}}  t_{p_{m+1}}$.

This completes the proof of the claim.\\

We next show that the subword of $w$ starting at $b_{p_m}  t_{p_m}$ and ending at $b_{p_{m+1}}  t_{p_{m+1}}$ ($1\leq m\leq j-1$) belongs to the set of words
generated from the expression
\begin{align}\label{form2}
\{b_{p_m} t_{p_m}\}^{+} \{(b_{p_m}-1)  t_{p_m}\}^{*} \cdots
\{b_{p_{m+1}} t_{p_m}\}^{*} \{b_{p_{m+1}} (t_{p_m}-1)\}^{*}\cdots\{b_{p_{m+1}}  t_{p_{m+1}}\}^{+}.
\end{align}
Indeed, if a pair $b_xt_x$ is between $b_{p_m}  t_{p_m}$ and $b_{p_{m+1}}  t_{p_{m+1}}$ in $w$, then combining the definition of a cut-pair with Lemma~\ref{lem-decreasing}, there is
$b_{p_m}>b_x\geq b_{p_{m+1}} \mathrm{~and~}t_{p_m}\geq t_x> t_{p_{m+1}}$. Together with the claim above, we have either
$b_{p_m}>b_x> b_{p_{m+1}}$, $t_x=t_{p_m}$,
or $b_x=b_{p_{m+1}}$, $t_{p_m}\geq t_x>t_{p_{m+1}}$. Thus, for any two distinct pairs $b_xt_x$ and $b_yt_y$ between $b_{p_m}  t_{p_m}$ and $b_{p_{m+1}}  t_{p_{m+1}}$ in $w$, there are three cases to consider. \\

\noindent {\bf Case 1:} $b_{p_m}>b_x> b_{p_{m+1}}$, $t_x=t_{p_m}$, and $b_{p_m}>b_y> b_{p_{m+1}}$, $t_y=t_{p_m}$. Then $b_xt_x$ is to the left of $b_yt_y$ in $w$ if $b_x>b_y$ and to the right of it otherwise. \\

\noindent {\bf Case 2:} $b_{p_m}>b_x> b_{p_{m+1}}$, $t_x=t_{p_m}$, and $b_y=b_{p_{m+1}}$,
$t_{p_m}\geq t_y>t_{p_{m+1}}$. Then $b_xt_x$ is to the left of $b_yt_y$ in $w$. \\

\noindent {\bf Case 3:} $b_x=b_{p_{m+1}}$, $t_{p_m}\geq t_x>t_{p_{m+1}}$, and $b_y=b_{p_{m+1}}$,
$t_{p_m}\geq t_y>t_{p_{m+1}}$. Then $b_xt_x$ is to the left of $b_yt_y$ in $w$ if $t_x>t_y$ and to the right of it otherwise. \\

\noindent Hence, it follows that the subword of $w$ between $b_{p_m}  t_{p_m}$ and $b_{p_{m+1}}   t_{p_{m+1}}$ ($1\leq m\leq j-1$) can be generated by an expression of the form \eqref{form2}.

Similarly, if there is a pair $b_xt_x$ to the left of $b_{p_1}  t_{p_1}$ in $w$, then $k> b_x>b_{p_1}$ and $  k>t_x >t_{p_1}$
can not happen at the same time.
That is to say, there is
$k>b_x> b_{p_1}$ and $t_x=k$, or $b_x=b_{p_1}$ and $k\geq t_x>t_{p_1}$.
The subword of $w$ to the left of $b_{p_1}  t_{p_1}$  belongs to the set of words
that can be generated by the expression
\begin{align*}
\{(k-1)k\}^{*}\cdots\{b_{p_1} k\}^{*} \{b_{p_1}(k-1)\}^{*}\cdots\{b_{p_1}  t_{p_1}\}^{+}.
\end{align*}

And similarly, if there is a pair $b_xt_x$ is to the right of $b_{p_j}  t_{p_j}$ in $w$, then $b_{p_j}> b_x>1$ and $  t_{p_j}>t_x >1$
cannot happen at the same time.
That is to say, there is
$b_{p_j}>b_x>1 $ and $t_x= t_{p_j}$, or $b_x=1$ and $t_{p_j}\geq t_x>t_{p_1}$.
The subword of $w$ after $b_{p_j}  t_{p_j}$  belongs to the set of words that can be
generated by the expression
\begin{align*}
\{b_{p_j}  t_{p_j}\}^{+}\{(b_{p_j}-1) t_{p_j}\}^{*}\cdots \{1 t_{p_j}\}^{*}\cdots\{12\}^{*}.
\end{align*}

Thus, we obtain that every word $w\in S^{123}_{k,2i}$
with cut-pairs  $b_{p_1}  t_{p_1}, b_{p_2}  t_{p_2}, \ldots,  b_{p_j}  t_{p_j}$, where $p_1<p_2<\cdots< p_j$, belongs to the set of words that can be generated by the expression~\eqref{form_2}.

Finally, two different expressions of the form \eqref{form_2}  cannot produce the same word since they belong to two different cut-equivalence classes.
This completes the proof.
\end{proof}

From Lemma \ref{lem-cutpair}, we see that each cut-equivalence class in $S^{123}_{k,2i}$ can be represented by an expression of the form (\ref{form_2}). For example, given $k=5$, there are five solutions to $4>b_{p_1} >\cdots > b_{p_j}>1$ and $5>t_{p_1} > \cdots > t_{p_j}>2$ with $b_{p_s}<b_{p_s}$ for $1\leq s\leq j$. Indeed, it is not difficult to see that $0\leq j\leq 2$. When $j=0$, the solution is the empty set; when $j=1$, the three solutions are $\{23\}$, $\{24\}$ and $\{34\}$; when $j=2$, the unique solution is $\{34,23\}$. Note that each solution corresponds to a cut-equivalence class with the corresponding cut-pairs. Thus
$\mathcal{F}^{123}_{5,2i}$, the set of cut-equivalence classes for  $S^{123}_{5,2i}$, is as follows:

{\bf Class 1:}  $\{45\}^{*}\{35\}^{*}\{25\}^{*}\{15\}
^{*}\{14\}^{*}\{13\}^{*}\{12\}^{*}$;

{\bf Class 2:} $\{45\}^{*}\{35\}^{*}\{25\}^{*}\{24\}^{+}\{14\}^{*}\{13\}^{*}\{12\}^{*}$;

{\bf Class 3:} $\{45\}^{*}\{35\}^{*}\{25\}^{*}\{24\}^{*}\{23\}^{+}\{13\}^{*}\{12\}^{*}$;

{\bf Class 4:} $\{45\}^{*}\{35\}^{*}\{34\}^{+}\{24\}^{*}\{14\}^{*}\{13\}^{*}\{12\}^{*}$;

{\bf Class 5:} $\{45\}^{*}\{35\}^{*}\{34\}^{+}\{24\}^{*}\{23\}^{+}\{13\}^{*}\{12\}^{*}$.

\subsection{A bijection between Dyck paths  and cut-equivalence classes}

Let $\mathbf{D}_n$ denote the set of all Dyck paths of semi-length $n$. It is a well-known fact that the number of paths in $\mathbf{D}_n$  is given by $C_n$, the $n$-th Catalan number.

Each Dyck path in $\mathbf{D}_n$ can be encoded by a {\em Dyck word} $\pi=\pi_1\pi_2\cdots \pi_{2n}$, where $\pi_i\in \{U,D\}$ for $1\leq i\leq 2n$, and $\pi$ satisfies the condition that for  $1\leq k\leq 2n$,  the number of $U$s in $\pi_1\pi_2\cdots \pi_k$ is no less than the number of $D$s there. Thus, $U$ corresponds to an up-step $(1,1)$ and $D$ corresponds to a down-step $(1,-1)$. Slightly abusing the terminology, we think of a Dyck path to be the same as the Dyck word encoding it.

A {\em valley} in $\pi\in \mathbf{D}_n$ is an occurrence of $DU$, that is, a $U$ in $\pi$ immediately preceded by a $D$.  We let $v(\pi)$ denote the number of valleys in $\pi$. For example, Figure~\ref{dyckpathexample} shows a Dyck path $\pi$ of semi-length 8 with $v(\pi)=3$. It is a well-known result that the number of  Dyck paths of semi-length $n$ with $j$ valleys is given the Narayana number $N_{n,j}=\frac{1}{j+1}\binom{n}{j} \binom{n-1}{j}$, where $0\leq j \leq n-1$.

\begin{figure}[htbp]
  \centering
  \begin{tikzpicture}[scale=0.5]
  \draw[line width=0.02] (0, 0) grid (16, 4);
  \draw[rounded corners=1, color=black, line width=1] (0, 0) -- (1, 1) -- (2, 2) -- (3, 1) -- (4, 0) -- (5, 1) -- (6, 2) -- (7, 3) -- (8, 4) -- (9, 3) -- (10, 2) -- (11, 1) -- (12, 2) -- (13, 1) -- (14, 2) -- (15, 1) -- (16, 0);
\end{tikzpicture}
  \caption{The Dyck path $\pi=UUDDUUUUDDDUDUDD$.}
  \label{dyckpathexample}
\end{figure}
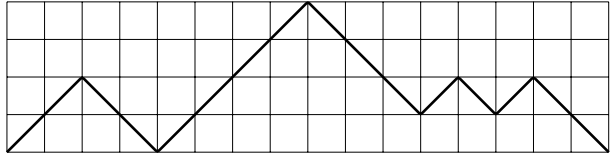

Given a cut-equivalence class in $S^{123}_{k,2i}$, we define a Dyck path as follows: start at the point $(0,0)$ and go along an up-step. Then if the pair immediately after $(k-2)k$ is $(k-3) k$, go along an up-step, while if it is $(k-2)(k-1)$, go along a down-step. In general, if in the following pair  the bottom element is decreased by 1, go along an up-step, while if the top element there is decreased by 1, go along a down-step. See Figure \ref{bbbb} for an example when $k=5$. Note that cut-pairs in the correspondence given by Figure \ref{bbbb} correspond to valleys, which is not a coincidence. This correspondence leads to the following theorem, the main result of this subsection.

\begin{figure}[htbp]
  \centering
\begin{tikzpicture}[scale=0.4]

%
%
%
%
%

\draw [line width=0.6]
(30,0)--(32,2)--(34,4)--(36,6)--(38,4)--(40,2)--(42,0);
\draw [line width=0.6]
(30,0)--(32,2)--(34,4)--(36,2)--(38,4)--(40,2)--(42,0);
\draw [line width=0.6]
(30,0)--(32,2)--(34,4)--(36,2)--(38,0)--(40,2)--(42,0);
\draw [line width=0.6]
(30,0)--(32,2)--(34,0)--(36,2)--(38,4)--(40,2)--(42,0);
\draw [line width=0.6]
(30,0)--(32,2)--(34,0)--(36,2)--(38,0)--(40,2)--(42,0);

\node [right] at (30,0){$45$};\node [right] at (34,0) {$34$};
\node [right] at (38,0){$23$};\node [right] at (42,0) {$12$};

\node [right] at (32,2){$35$};\node [right] at (36,2) {$24$};
\node [right] at (40,2){$13$};

\node [right] at (34,4) {$25$};\node [right] at (38,4){$14$};

\node [right] at (36,6) {$15$};

\end{tikzpicture}
\caption{The correspondence between cut-equivalence classes and  Dyck paths.}
 \label{bbbb}
\end{figure}
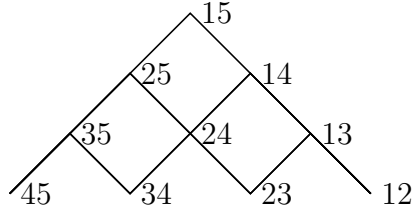

\begin{thm}\label{class_dyck}
There is a bijection $\phi$ from $\mathcal{F}^{123}_{k,2i}$ to $\mathbf{D}_{k-2}$ such that if $f\in \mathcal{F}^{123}_{k,2i}$ and $\phi(f)=\pi$ then the number of cut-pairs $n(f)$ is equal to the number of valleys $v(\pi)$.
\end{thm}

\begin{proof}
Let $f$ be the cut-equivalence class in $\mathcal{F}^{123}_{k,2i}$ with cut-pairs $b_{p_1}  t_{p_1}$, $b_{p_2}  t_{p_2}$, $\ldots$,  $b_{p_j}  t_{p_j}$.
We define $\pi=\phi(f)$ to be
\begin{align}\label{1}
\underbrace{U\cdots U}_{k-1-b_{p_1}}
\underbrace{D\cdots D}_{k-t_{p_1}}
\underbrace{U\cdots U}_{b_{p_1}-b_{p_2}}
\underbrace{D\cdots D}_{t_{p_1}-t_{p_2}}
\cdots\cdots
\underbrace{U\cdots U}_{b_{p_{j-1}}-b_{p_j}}
\underbrace{D\cdots D}_{t_{p_{j-1}}-t_{p_j}}
\underbrace{U\cdots U}_{b_{p_j}-1}
\underbrace{D\cdots D}_{t_{p_j}-2}.
\end{align}
In particular, if $f$ is the unique cut-equivalence class containing no cut-pairs, then
\begin{align*}
\phi(f)=\underbrace{U\cdots U}_{k-2}
\underbrace{D\cdots D}_{k-2}.
\end{align*}
Clearly, $\pi$ contains $k-2$ up-steps and $k-2$ down-steps.
Moreover, since $b_{p_{\ell}} < t_{p_{\ell}}$ for $1\le \ell \le j$, we have that $k-1-b_{p_{\ell}} \geq k-t_{p_{\ell}}$ for $1\le \ell \le j$, which implies that the number of up-steps is never less than that of down-steps in any initial part of $\pi$. Thus, $\pi\in \mathbf{D}_{k-2}$. Finally, note that $\pi$ contains exactly $j$ valleys since there are $j$ $DU$s in $\pi$, and thus $n(f)=v(\pi)$.

In order to show that $\phi$ is injective, we need to show that for different $f_1,f_2\in\mathcal{F}^{123}_{k,2i}$, we have $\phi(f_1)\neq\phi(f_2)$.
If $n(f_1)\neq n(f_2)$, then $v(\phi(f_1))\neq v(\phi(f_2))$ and thus $\phi(f_1)\neq\phi(f_2)$. If $n(f_1)=n(f_2)=j$, where $1\leq j\leq k-3$,
suppose that the cut-pairs of $f_1$ are $b_{p_1}  t_{p_1}$, $b_{p_2}  t_{p_2}$, $\ldots$,  $b_{p_j}  t_{p_j}$ and the cut-pairs of $f_2$ are $b^{'}_{p_1}  t^{'}_{p_1}$, $b^{'}_{p_2}  t^{'}_{p_2}$, $\ldots$,  $b^{'}_{p_j}  t^{'}_{p_j}$.
Let $j^*$ be the smallest index such that $b_{p_{j^{*}}}  t_{p_{j^{*}}}\neq b^{'}_{p_{j^{*}}}  t^{'}_{p_{j^{*}}}$, so that for any $j^{**}$, $1\leq j^{**} <j^{*}$, we have $b_{p_{j^{**}}}  t_{p_{j^{**}}}= b^{'}_{p_{j^{**}}}  t^{'}_{p_{j^{**}}}$. According to the definition of $\phi$, $\phi(f_1)$ and $\phi(f_2)$ are the same in the first $k-1-b_{p_{j^{*}-1}}$ up-steps and the first $k-t_{p_{j^{*}-1}}$ down-steps. Then in $\phi(f_1)$, $b_{p_{j^{*}-1}}-b_{p_{j^{*}}}$ up-steps and $t_{p_{j^{*}-1}}-t_{p_{j^{*}}}$ down-steps follow, and in $\phi(f_2)$, $b_{p_{j^{*}-1}}-b^{'}_{p_{j^{*}}}$ up-steps and $t_{p_{j^{*}-1}}-t^{'}_{p_{j^{*}}}$ down-steps follow. However, because either $b_{p_{j^{*}-1}}-b_{p_{j^{*}}}\neq b_{p_{j^{*}-1}}-b^{'}_{p_{j^{*}}}$ or $t_{p_{j^{*}-1}}-t_{p_{j^{*}}}\neq t_{p_{j^{*}-1}}-t^{'}_{p_{j^{*}}}$, we have that $\phi(f_1)\neq\phi(f_2)$.

To complete the proof, it remains to describe the inverse map $\phi^{-1}$.
For any Dyck path $\pi\in \mathbf{D}_{k-2}$ with $v(\pi)=j$, $\pi$ must be of the form
\begin{align}\label{2}
\underbrace{U\cdots U}_{\alpha_1}
\underbrace{D\cdots D}_{\beta_1}
\underbrace{U\cdots U}_{\alpha_2}
\underbrace{D\cdots D}_{\beta_2}
\cdots\cdots
\underbrace{U\cdots U}_{\alpha_j}
\underbrace{D\cdots D}_{\beta_j}
\underbrace{U\cdots U}_{\alpha_{j+1}}
\underbrace{D\cdots D}_{\beta_{j+1}},
\end{align}
where $\alpha_m>0$ and $\beta_m>0$ for $1\leq m\leq j+1$,  and $\sum_{i=1}^{j+1}\alpha_{i}=\sum_{i=1}^{j+1}\beta_{i}=k-2$.
We define the corresponding cut-equivalence class as follows. For $1\leq m\leq j$,
let $$b_{p_m}=k-1-(\alpha_1+\alpha_2+\cdots +\alpha_m)$$ and
$$t_{p_m}=k-(\beta_1+\beta_2+\cdots +\beta_m).$$
It is clear that $k-1>b_{p_1} > b_{p_2} >\cdots > b_{p_j}>1$ and
$k>t_{p_1} > t_{p_2} >\cdots > t_{p_j}>2$.
By Lemma \ref{lem-cutpair}, the cut-pairs of a 123-avoiding up-down word uniquely determine the cut-equivalence class that it belongs to. Thus, we can determine the cut-equivalence class $f$ corresponding to the Dyck path $\pi$ from the sequence of integer pairs  $\{(b_{p_m},  t_{p_m}) \}_{m=1}^{j}$. Clearly, we have $n(f)=j$.

Moreover, combining forms (\ref{1}) and (\ref{2}), we can get that $\phi\circ\phi^{-1}=\phi^{-1}\circ\phi=id$.
This completes the proof.
\end{proof}

To illustrate the bijection given in Theorem~\ref{class_dyck},  we consider the set $S^{123}_{5,2i}$ whose five cut-equivalence classes were listed above. The Dyck paths corresponding to these classes, in the respective order, are given in Figure~\ref{aaaaa}.
{\bf Class 1} is the only class in  $S^{123}_{5,2i}$ which has no cut-pair.
{\bf Classes 2, 3} and {\bf 4} have one cut-pair. The only class in  $S^{123}_{5,2i}$ which has two cut-pairs is {\bf Class 5}.

\begin{figure}[htbp]
  \centering
\begin{tikzpicture}[scale=0.4]

\draw[line width=0.005] (8, 0) grid (14,3);
\draw (8,0) -- +(6,0);
\draw [line width=0.5]
(8,0)--(9,1)--(10,2)--(11,3)--(12,2)--(13,1)--(14,0);

\draw[line width=0.005] (16, 0) grid (22,3);
\draw (16,0) -- +(6,0);
\draw [line width=0.6]
(16,0)--(17,1)--(18,2)--(19,1)--(20,2)--(21,1)--(22,0);
\draw (19,1) node [scale=0.25, circle, draw,fill=red]{};

\draw[line width=0.005] (24, 0) grid (30,3);
\draw (24,0) -- +(6,0);
\draw [line width=0.6]
(24,0)--(25,1)--(26,2)--(27,1)--(28,0)--(29,1)--(30,0);
\draw (28,0) node [scale=0.25, circle, draw,fill=red]{};

\draw[line width=0.005] (32, 0) grid (38,3);
\draw (32,0) -- +(6,0);
\draw [line width=0.6]
(32,0)--(33,1)--(34,0)--(35,1)--(36,2)--(37,1)--(38,0);
\draw (34,0) node [scale=0.25, circle, draw,fill=red]{};

\draw[line width=0.005] (40, 0) grid (46,3);
\draw (40,0) -- +(6,0);
\draw [line width=0.5]
(40,0)--(41,1)--(42,0)--(43,1)--(44,0)--(45,1)--(46,0);
\draw (42,0) node [scale=0.25, circle, draw,fill=red]{};
\draw (44,0) node [scale=0.25, circle, draw,fill=red]{};
\end{tikzpicture}
 \caption{Dyck paths corresponding to cut-equivalence {\bf Classes 1--5} in $S^{123}_{5,2i}$, respectively.}
 \label{aaaaa}
\end{figure}
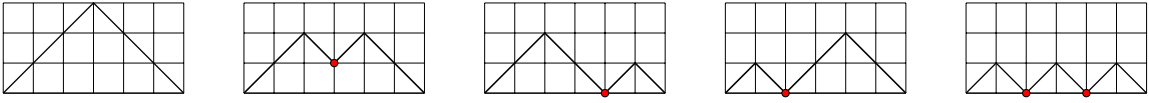

The following statement is an immediate corollary to Theorem~\ref{class_dyck} and well-known enumerative properties of Dyck paths.

\begin{cor}\label{lem:cut}
There are $C_{k-2}$
 equivalence classes with respect to the cut-equivalence relation in  $S^{123}_{k,2i}$.
Moreover, the number of  cut-equivalence classes with $j$ cut-pairs in  $S^{123}_{k,2i}$ is $N_{k-2,j}$,  where $0\leq j\leq k-3$.
\end{cor}

\subsection{An alternative enumeration of \texorpdfstring{$N^{123}_{k,2i}$}{Lg}}

Corollary \ref{lem:cut} allows us to give an alternative, combinatorial proof of the following theorem appearing in~\cite{GKZ}.

\begin{thm}[\cite{GKZ}]\label{thm-alt-comb-proof} For $k\ge 3$, we have
$$N^{123}_{k,2i}=\frac{1}{i+1} \binom{i+k-2}{i} \binom{i+k-1}{i}.$$
\end{thm}

\begin{proof}
Let $f$ be the cut-equivalence class corresponding to cut-pairs
$b_{p_1}  t_{p_1}$, $b_{p_2}  t_{p_2}$, $\ldots$,  $b_{p_j}  t_{p_j}$.
We first claim that the number of words of length $2i$ belonging to  $f$ is
$\binom{2 k-4+i-j}{2 k-4}$. Indeed, by Lemma~\ref{lem-cutpair}, any word $w\in f$ must be obtained from (\ref{form_2}). Further, by Theorem~\ref{class_dyck}, there are at most $2(k-2)+1=2k-3$ distinct pairs in (\ref{form_2}), which gives an upper bound on the number of distinct pairs in $\mathcal{P}_w$.
For $1\leq i\leq 2k-3$, we let $x_i$ denote the number of times the $i$-th pair in (\ref{form_2}), from left to right, appears in $w$. Thus the words in $f$ are in 1-to-1 correspondence with nonnegative solutions of the equation $x_1+x_2+\cdots+x_{2k-3}=i$, where $j$ specified $x_m$s (corresponding to cut-pairs) are forced to be positive. The number of such solutions is $\binom{2 k-4+i-j}{2 k-4}$, as desired.

Combining the last statement with Corollary~\ref{lem:cut}, we obtain that
\begin{align*}
N^{123}_{k,2i} & =
\sum _{j=0}^{k-3}N_{k-2,j}\binom{2 k-4+i-j}{2 k-4}.
\end{align*}
In what follows, we shall give a closed form formula for $N^{123}_{k,2i}$. We start with using the formula for the Narayana numbers $N_{k-2,j}$ to obtain 
\begin{align*}
N^{123}_{k,2i} & = \sum _{j=0}^{k-3} \frac{1}{j+1} \binom{k-3}{j} \binom{k-2}{j} \binom{2 k-4+i-j}{2 k-4}.
\end{align*}
Since the factor $\binom{k-3}{j}$ vanishes for $j>k-3$, we have that 
\begin{align}\label{eq:N}
N^{123}_{k,2i} & = \sum _{j=0}^{\infty} \frac{1}{j+1} \binom{k-3}{j} \binom{k-2}{j} \binom{2 k-4+i-j}{2 k-4}.
\end{align}
We next use the approach described in~\cite[p. 35]{Petkovsek1996AB} to express \eqref{eq:N} in terms of  a hypergeometric series. 
Denote $(a)_n$ the rising factorial $a(a+1)\cdots (a+n-1)$ and let $\pFq{3}{2}{a_1,a_2,a_3}{b_1,b_2}{z}$ be
$$\sum_{s = 0}^{\infty} \frac{(a_1)_s (a_2)_s (a_3)_s}{(b_1)_s (b_2)_s s!}z^s.$$
Since the constant coefficient of \eqref{eq:N} is $\binom{i+2 k-4}{2 k-4}$ and the ratio between consecutive coefficients in \eqref{eq:N} is
\begin{align*}
 \frac{[x^{j+1}]N^{123}_{k,2i}}{[x^{j}]N^{123}_{k,2i}} =  \frac{(3-k+j) (2-k+j) (-i+j)}{(2+j) (4-i-2k+j) (1+j)},
\end{align*}
it follows that
$N^{123}_{k,2i}$ can be expressed as 
\begin{align}\label{eq:3F2}
 \binom{i+2 k-4}{2 k-4} \times 
 \pFq{3}{2}{3-k,2-k,-i}{2,4-i-2k}{1}.
\end{align}
The {S}aalsch\"utz  identity \cite[p. 9]{Bailey1964Generalized} says that 
\begin{align*}
\pFq{3}{2}{a,b,-n}{c,1+a+b-c-n}{1}=\frac{(c-a)_n (c-a)_n}{(c)_n (c-a-b)_n}.
\end{align*}
By setting $a=3-k$, $b=2-k$, $c=2$ and $i=n$, we have that 
\begin{align*}
\pFq{3}{2}{3-k,2-k,-i}{2,4-i-2k}{1}
 = \frac{(k-1)_i (k)_i}{(2)_i (2k-3)_i }
 = \frac{(i+k-2)! (i+k-1)! (2k-4)!}{(k-2)! (k-1)! (i+1)! (i+2k-4)!}.
\end{align*}
Substituting the last formula into \eqref{eq:3F2}, we obtain that 
\begin{align*}
 N^{123}_{k,2i} &= \frac{1}{i+1}\binom{i+k-2}{i} \binom{i+k-1}{i},
\end{align*}
which completes the proof.
\end{proof}

To illustrate Theorem~\ref{thm-alt-comb-proof}, the words in the five cut-equivalence classess in  $S^{123}_{5,2i}$ are enumerated by $\binom{i+6}{6}$, $\binom{i+5}{6}$, $\binom{i+5}{6}$, $\binom{i+5}{6}$, and $\binom{i+4}{6}$, respectively.
Hence,  the number of words in  $S^{123}_{5,2i}$ is $$\binom{i+6}{6}+3\binom{i+5}{6}+\binom{i+4}{6}= \frac{1}{i+1}\binom{i+4}{4}\binom{i+3}{3}.$$

\section{Enumeration of length 3 consecutive pattern-avoiding up-down words}\label{three-consecutive-up-down-sec}

The following theorem is a straightforward corollary to Formula (1) in~\cite{GKZ}, since the patterns $\underline{123}$ and $\underline{321}$ do not bring any new restrictions on alternating words, and thus $S^{\underline{123}}_{k,\ell}=S^{\underline{321}}_{k,\ell}$ is the set of up-down words of length $\ell$ over $[k]$. In what follows, $\delta_{a,b}$ is the\emph{ Kronecker delta}, which is equal to $1$ if $a=b$ and $0$ otherwise. Also, $\chi(a)$ equals $1$ if $a$ is true, and $0$ otherwise.

\begin{thm}\label{trivial-case} We have
$$N^{\underline{123}}_{k,\ell}=N^{\underline{321}}_{k,\ell}=M_{k,\ell},$$
where the numbers $M_{k,\ell}$ satisfy the following recurrence relation for $k\ge 3$ and $\ell\ge 2$:
\begin{align}\label{eq-N-rec}
 M_{k,\ell} = M_{k-1,\ell} + \sum_{i=0}^{\lfloor \frac{\ell-1}{2}\rfloor} M_{k-1,2i} M_{k,\ell-2i-1} - \chi(\ell\mathrm{~is~ even~}) \cdot M_{k-1,\ell-2}
\end{align}
with the initial conditions $M_{k, 0}=1$, $M_{k, 1}=k$ for $k\ge 2$, and $M_{2,\ell}=1$ for $\ell \ge 2$.
\end{thm}
%

\subsection{\texorpdfstring{$\underline{132}$}{Lg}-avoiding up-down words}\label{sub-cons-132}

Table~\ref{tab1} provides the numbers $N_{k,\ell}^{\underline{132}}$ of $\underline{132}$-avoiding up-down words of length $\ell$ over an alphabet $[k]$ for small values of $k$ and $\ell$.
For convenience, we present separately even and odd length cases.

\begin{table}[!htb]
    \centering
    \begin{minipage}{.5\textwidth}
        \centering
    \begin{tabular}{|c|lllll|l|}
\hline
\diagbox{$k$}{$\ell$}& 0 & 2  & 4  &6 &8\\
\hline
2&  1& 1& 1&1&1\\
\hline
3&  1&3& 7&  15&  31\\
\hline
4&  1&  6& 25&  90& 301\\
\hline
5&  1&  10&  65&  350& 1701\\
\hline
\end{tabular}
    \end{minipage}%
    \begin{minipage}{0.5\textwidth}
        \centering
\begin{tabular}{|c|lllll|l|}
\hline
\diagbox{$k$}{$\ell$}&  1  & 3  &5 & 7&9\\
\hline
2&   2& 1& 1& 1&1\\
\hline
3&   3& 4& 8& 16&32\\
\hline
4&    4&  10&  33&  106 &333\\
\hline
5&   5& 20&  98&456&2034\\
\hline
\end{tabular}
    \end{minipage}
  \caption{$N_{k,\ell}^{\underline{132}}$ for small values of $k$ and~$\ell$.}
\label{tab1}
\end{table}

\begin{lem}\label{132}
An up-down word $w=w_1w_2\cdots w_{\ell}$
 is $\underline{132}$-avoiding if and only if the bottom elements of $w$ are weakly decreasing from left to right, i.e.,
$$b_1\geq b_2\geq\dots \geq b_{\lceil\frac{\ell}{2}\rceil}.$$
\end{lem}

\begin{proof}
If there were some $j$, $1\leq j\leq \lceil\frac{\ell}{2}\rceil-1$, such that $b_{j} < b_{j+1}$, then $b_jt_jb_{j+1}$ would form an occurrence of the pattern $\underline{132}$.

Conversely, if there is an occurrence $w_jw_{j+1}w_{j+2}$ of the pattern $\underline{132}$  in $w$, where $1\leq j\leq \ell-2$, then we have $w_j< w_{j+2} < w_{j+1}$. According to the definition of up-down words, $w_{j+1}$ must be a top element in $w$, and $w_j$ and $w_{j+2}$ must be bottom elements in $w$, and $w_j< w_{j+2}$.

This completes the proof.
\end{proof}
Let $A_{k,\ell}=N_{k,\ell}^{\underline{132}}$. Next theorem enumerates $A_{k,2i}$.

\begin{thm}\label{thm 1.4}
For all $k\geq 2$ and $i\geq 0$, we have $$A_{k,2i}=S(k+i-1,k-1),$$ where $S(n,m)$ is a Stirling number of the second kind.
\end{thm}

\begin{proof}

Note that for $k\geq 3$ and $i\geq 1$, any $\underline{132}$-avoiding up-down word $w$ of length $2i$ over $[k]$, belongs to one of the following two cases:
\begin{itemize}
\item[(a)] There are no $1$s in $w$. These words are counted by $A_{k-1,2i}$ (which can be seen by subtracting a 1 from each element in $w$);
\item[(b)] There is at least one $1$ in $w$.  By Lemma \ref{132}, $w_{2i-1}=1$, since $w_{2i-1}$ is the minimum element in $w$. Thus $w$ is of the form $w'1w''$, where $w'$ is a $\underline{132}$-avoiding up-down word of length $2i-2$ and $w''$ is a letter in $\{2,3,\ldots,k\}$. Such words are counted by $(k-1)A_{k,2i-2}$.
\end{itemize}
Hence for $k\geq 3$ and $i\geq 1$, the numbers  $A_{k,2i}$ satisfy the recurrence relation
\begin{align}\label{rec}
A_{k,2i}
=A_{k-1, 2i}+(k-1)A_{k,2i-2}
\end{align}
with the initial  conditions $A_{2,2i}= 1$ for all $i\geq1$ and $A_{k,0}=1$ for all $k \geq 2$, which are easy to check.

We have that $A_{k,2i}=S(k+i-1,k-1)$ since these numbers have the same recurrence relation and initial conditions. Indeed, from a well-known recurrence relation for the Stirling numbers of the second kind,
$$S(k+i-1,k-1)=S(k+i-2,k-2)+(k-1)S(k+i-2,k-1),$$
together with their initial conditions $S(i+1,1)= 1$ for all $i\geq 0$ and $S(k-1,k-1)=1$ for all $k \geq 2$.
\end{proof}

We now turn our attention to considering $A_{k,2i+1}$.

\begin{thm}\label{thm 1.6}
For all $k\geq 2$ and $i\geq 1$, we have
\begin{align}
A_{k,2i+1}=\sum_{j=2}^{k}A_{j,2i}.
\end{align}
\end{thm}

\begin{proof}
Let $A_{k, \ell}^j$ denote the number of those words counted by $A_{k,\ell}$ that end with $j$.
It is easy to see that for $k\geq 2$ and  $i\geq1$,
\begin{align*}
 A_{k,2i+1}=\sum_{j=1}^{k-1} A_{k,2i+1}^j.
\end{align*}
By Lemma \ref{132}, for any word $w \in S^{\underline{132}}_{k,2i+1}$ whose last letter is $j$, the minimum letter of $w$ is also $j$. Thus, we have that $$A_{k,2i+1}^j=A_{k-j+1,2i+1}^1,$$ where $1\leq j\leq k-1$, because we can subtract $j$ from each letter of any word counted by $A_{k,2i+1}^j$. Moreover, for any word in $S^{\underline{132}}_{k-j+1,2i+1}$ ending with $1$, we can remove $1$ to form a word of length $2i$, which is also \underline{132}-avoiding. On the other hand, for any word  $S^{\underline{132}}_{k-j+1,2i}$, we can adjoin the letter $1$ at the end to form a $\underline{132}$-avoiding word of length $2i+1$. Thus, $$A_{k-j+1,2i+1}^1=A_{k-j+1,2i}.$$
So,  we obtain that
$$A_{k,2i+1}=\sum_{j=1}^{k-1}A_{k-j+1,2i}=\sum_{j=2}^{k}A_{j,2i}.$$
\end{proof}

\begin{thm} For $k \geq 2$, let $N_k^{\underline{132}}(x)=\sum_{\ell\geq 0}A_{k,\ell}x^{\ell}$ be the generating function for $N_{k,\ell}^{\underline{132}}$. Then we have
$$N_k^{\underline{132}}(x)=\sum_{j=1}^{k}\frac{x+\delta_{j,k}}{(1-x^2)(1-2x^2)\cdots(1-(j-1)x^2)}.$$
\end{thm}

\begin{proof}
Let $$A_k(x)=\sum_{i\geq 0}A_{k,2i}x^i.$$
By (\ref{rec}), it follows that
\begin{align*}
 A_k(x)& =\sum_{i\geq 0}A_{k,2i}x^i\\
 &=1+\sum_{i\geq1}A_{k-1,2i}x^i+(k-1)\sum_{i\geq 1}A_{k,2i-2}x^i\\
 &=A_{k-1}(x)+(k-1)xA_k(x)
\end{align*}
for $k\geq 2$ and $A_1(x)=1$.
This leads to the following well-known generating function for Stirling numbers of the second kind, where $k\geq 1$:
\begin{align*}
A_k(x)&=\frac{1}{(1-x)(1-2x)\cdots(1-(k-1)x)}.
\end{align*}

From the definition of $N_k^{\underline{132}}(x)$ as well as the fact $A_{k,1} = k$, we have that
\begin{align*}
 N_k^{\underline{132}}(x)& =\sum_{\ell\geq 0}A_{k,\ell}x^{\ell}\\
 &=\sum_{i\geq0}A_{k,2i}x^{2i}+\sum_{i\geq 0}A_{k,2i+1}x^{2i+1}\\
 &=A_{k}(x^2)+\sum_{i\geq 0}\sum_{j=2}^{k}A_{j,2i}x^{2i+1}+x\\
 &=A_{k}(x^2)+x\sum_{j=2}^{k}\sum_{i\geq 0}A_{j,2i}x^{2i}+x\\
 &=A_{k}(x^2)+x\sum_{j=2}^{k}A_{j}(x^2)+x\\
 &=\sum_{j=1}^{k}\frac{x+\delta_{j,k}}{(1-x^2)(1-2x^2)\cdots(1-(j-1)x^2)},
\end{align*}
as desired. This completes the proof.
\end{proof}


\subsection{\texorpdfstring{$\underline{312}$}{Lg}-avoiding up-down words}\label{sec-3-2}

In this subsection, we consider the enumeration of $\underline{312}$-avoiding up-down words, which is similar  to the enumeration of $\underline{132}$-avoiding up-down words done in Section~\ref{sub-cons-132}.
Table~\ref{tab2} provides the numbers $N_{k,\ell}^{\underline{312}}$ for small values of $k$ and $\ell$.

\begin{table}[!htb]
\centering
    \begin{minipage}{.5\textwidth}
        \centering
    \begin{tabular}{|c|lllll|l|}
\hline
\diagbox{$k$}{$\ell$}& 0 & 2  & 4  &6 &8\\
\hline
2&  1& 1& 1&1&1\\
\hline
3&  1&3& 6&  12&   24\\
\hline
4&  1&  6& 20&  65& 206\\
\hline
5&  1&  10&   50&  238& 1080\\
\hline
\end{tabular}
    \end{minipage}%
    \begin{minipage}{0.5\textwidth}
        \centering
\begin{tabular}{|c|lllll|l|}
\hline
\diagbox{$k$}{$\ell$}&  1  & 3  &5 & 7&9\\
\hline
2&   2& 1& 1& 1&1\\
\hline
3&   3& 5& 11& 23&47\\
\hline
4&    4&  14&  53&  182&593\\
\hline
5&   5& 30&  173&874&4089\\
\hline
\end{tabular}
    \end{minipage}
  \caption{$N_{k,\ell}^{\underline{312}}$ for small values of $k$ and~$\ell$.}
\label{tab2}
\end{table}

We begin with giving a description of $\underline{312}$-avoiding up-down words.
\begin{lem}\label{312}
An up-down word $w=w_1w_2\cdots w_{\ell}$
 is $\underline{312}$-avoiding if and only if the top elements of $w$ are weakly increasing from left to right, i.e.,
$$t_1\leq t_2\leq\dots \leq t_{\lfloor\frac{\ell}{2}\rfloor}.$$
\end{lem}

\begin{proof}
For any up-down word $w$, if there exists $1\leq j\leq \lfloor\frac{\ell}{2}\rfloor-1$ such that $t_{j} > t_{j+1}$, then $t_jb_{j+1}t_{j+1}$ would be an occurrence of the pattern $\underline{312}$.

Conversely, if there is an occurrence $w_jw_{j+1}w_{j+2}$ of the pattern $\underline{312}$ in $w$, where $1\leq j\leq \ell-2$, we would have $w_{j+1}< w_{j+2} < w_j$. By definition of up-down words, $w_{j+1}$ must be a bottom element in $w$, and $w_j$ and $w_{j+2}$ must be top elements in $w$. But then $w_j> w_{j+2}$. \end{proof}

For $\ell\geq 2$, let $B_{k,\ell}=N^{\underline{312}}_{k,\ell}$ denote the number of $\underline{312}$-avoiding up-down words  of length $\ell$ over an alphabet $[k]$. Also, let $B_{k,0}=1$ and to simplify our calculations, we assume that $B_{k,1}=k-1$.

First, we deal with the enumeration of $B_{k,2i+1}$.

\begin{thm}\label{thm 1.10}
For $k\geq2$ and $i\geq 1$, the numbers  $B_{k,2i+1}$ satisfy the recurrence relation
\begin{align}\label{rec-rel-for-B}
B_{k,2i+1}
=B_{k-1, 2i+1}+(k-1)B_{k,2i-1}
\end{align}
with the initial  conditions $B_{1,2i+1}= 0$ for all $i\geq1$ and $B_{k,1}=k-1$ for all $k \geq 2$.
Furthermore, if $$B_k(x)=\sum_{i\geq 0}B_{k,2i+1}x^i,$$ for $k\geq 1$, then $$B_k(x)=\sum_{j=1}^{k-1}\frac{1}{(1-jx)\cdots(1-(k-1)x)}.$$
\end{thm}

\begin{proof}
Our proof of (\ref{rec-rel-for-B}) is similar to the proof of (\ref{rec}) considering subclasses of whether $k$ appears in $w$ or not, and we omit it.

By (\ref{rec-rel-for-B}), we have
\begin{align*}
 B_k(x)& =\sum_{i\geq 0}B_{k,2i+1}x^i\\
 &=k-1+\sum_{i\geq1}B_{k-1,2i+1}x^i+(k-1)\sum_{i\geq 1}B_{k,2i-1}x^i\\
 &=1+B_{k-1}(x)+(k-1)xB_k(x)
\end{align*}
for $k\geq 2$.
Therefore, $$B_k(x)=\frac{B_{k-1}(x)+1}{1-(k-1)x}$$
with the initial condition $B_1(x)=0$.

Hence, for $k\geq 2$, we have
\begin{align*}
 B_k(x)& =\frac{1}{(1-x)(1-2x)\cdots(1-(k-1)x)}+\frac{1}{(1-2x)\cdots(1-(k-1)x)}+\cdots+
 \frac{1}{1-(k-1)x}\\
 &=\sum_{j=1}^{k-1}\frac{1}{(1-jx)\cdots(1-(k-1)x)},
\end{align*}
which completes the proof.
\end{proof}

Now we turn our attention to the words of even length.
\begin{thm}\label{thm 1.11}
For all $k\geq 2$ and $i\geq 2$, we have
 $$B_{k,2i}=\sum_{j=2}^{k}B_{j,2i-1}.$$
\end{thm}

\begin{proof}
Let $B_{k, \ell}^j$ denote the number of those words counted by $B_{k,\ell}$ that end with $j$ for $\ell\geq2$. It is easy to see that for $k\geq 2$ and $i\geq 2$,
\begin{align*}
 B_{k,2i}=\sum_{j=2}^{k}B_{k,2i}^j.
\end{align*}
For any word $w \in S^{\underline{312}}_{k,2i}$ whose last letter is $j$, by Lemma~\ref{312}, the maximum letter of $w$ is also $j$. Thus, for $2\leq j\leq k$, we have that $$B_{k,2i}^j=B_{j,2i}^j.$$

Moreover, for any word in $S^{\underline{312}}_{j,2i}$ ending with $j$, we can remove $j$ to form a word of length $2i-1$, which is also \underline{312}-avoiding. On the other hand, for any word in $S^{\underline{312}}_{j,2i-1}$, we can adjoin a letter $j$ at the end to form a $\underline{312}$-avoiding word of length $2i$. Thus, $$B_{j,2i}^j=B_{j,2i-1}.$$
So,  we obtain that
\begin{align*}
 B_{k,2i}& =\sum_{j=2}^{k}B_{j,2i-1},
\end{align*}
which completes the proof.
\end{proof}

\begin{prop}For $k \geq 2$, let $N_k^{\underline{312}}(x)=x+\sum_{\ell\geq 0}B_{k,\ell}x^{\ell}$ be the generating function for $N_{k,\ell}^{\underline{312}}$. Then
$$N_k^{\underline{312}}(x)=1+x+\sum_{j=2}^{k}\sum_{i=1}^{j-1}\frac{x^2+x
\delta_{j,k}}{(1-ix^2)\cdots(1-(j-1)x^2)}.$$
\end{prop}

\begin{proof}
From Theorem \ref{thm 1.11} together with the fact $B_{k,2}=\sum_{j=2}^{k}B_{j,1}=\binom{k}{2}$, we obtain that
\begin{align*}
 N_k^{\underline{312}}(x)&=x+\sum_{\ell\geq 0}B_{k,\ell}x^{\ell}\\
 &=x+\sum_{i\geq0}B_{k,2i}x^{2i}+\sum_{i\geq 0}B_{k,2i+1}x^{2i+1}\\
 &=1+x+\sum_{i\geq 1}\sum_{j=2}^{k}B_{j,2i-1}x^{2i}+xB_{k}(x^2)\\
 &=1+x+x^2\sum_{j=2}^{k}B_{j}(x^2)+xB_{k}(x^2)\\
 &=1+x+\sum_{j=2}^{k}\sum_{i=1}^{j-1}\frac{x^2+x\delta_{j,k}}{(1-ix^2)\cdots(1-(j-1)x^2)}.
\end{align*}
This completes the proof.
\end{proof}

%
%

\subsection{\texorpdfstring{$\underline{213}$}{Lg}-avoiding or \texorpdfstring{$\underline{231}$}{Lg}-avoiding up-down words}

In what follows, we resume using $N_{k,\ell}^p$ for the number of $p$-avoiding up-down words of length $\ell$ over an alphabet $[k]$.

\begin{thm} For all $k\geq 2$ and $i\geq 0$, we have
$$N_{k,2i+1}^{\underline{213}}=N_{k,2i+1}^{\underline{312}}$$
and
$$N_{k,2i+1}^{\underline{231}}=N_{k,2i+1}^{\underline{132}}.$$
\end{thm}

\begin{proof}
The equalities hold by applying the reverse operation to all words, which keeps the property of being an up-down word.
\end{proof}

For the case of the even lengths, we have the following result.
\begin{thm}
For all $k\geq 2$ and $i\geq 1$, there is
$$N_{k,2i}^{\underline{213}}=N_{k,2i}^{\underline{132}}$$
and
$$N_{k,2i}^{\underline{231}}=N_{k,2i}^{\underline{312}}.$$
\end{thm}
\begin{proof}
The statement follows by applying the complement and reverse operations which turn an up-down word into an up-down word.
\end{proof}

\section{Enumeration of up-down words avoiding a vincular pattern of length 3}\label{enum-vinc-pat}

In Section~\ref{three-consecutive-up-down-sec}, we enumerated up-down words avoiding consecutive patterns of length 3, which are a particular case of vincular patterns. In this section, we consider avoidance of other vincular patterns of length 3 on up-down words. We divide patterns of the form $x\underline{yz}$ into three subcases; in each subcase the proofs are similar.

\subsection{\texorpdfstring{$1\underline{32}$}{Lg}-avoiding or \texorpdfstring{$3\underline{12}$}{Lg}-avoiding up-down words}

Similarly to our considerations above, we first give a description of $1\underline{32}$-avoiding up-down words.

\begin{thm}\label{thm-132-av}The following two statements hold:
\begin{itemize}
\item[(a)]An up-down word $w=w_1w_2\cdots w_{\ell}$
 is $1\underline{32}$-avoiding if and only if the bottom elements of $w$ are weakly decreasing from left to right, i.e.,
$$b_1\geq b_2\geq\dots \geq b_{\lceil\frac{\ell}{2}\rceil}.$$
\item[(b)]An up-down word $w$
 is $1\underline{32}$-avoiding if and only if $w$ is $\underline{132}$-avoiding, and thus, for $k\geq 2$ and $\ell \geq 0$, we have
$$N_{k,\ell}^{1\underline{32}}=N_{k,\ell}^{\underline{132}},$$ which is enumerated in Section~\ref{sub-cons-132}.
\end{itemize}
\end{thm}

\proof

\begin{itemize}
\item[(a)] If there were some $j$, $1\leq j\leq \lceil\frac{\ell}{2}\rceil-1$, such that $b_{j} < b_{j+1}$, then $b_jt_jb_{j+1}$ would be an occurrence of the pattern $1\underline{32}$.

Conversely, if in $w$ there is an occurrence $w_{j^{*}}w_{j}w_{j+1}$ of the pattern $1\underline{32}$, where $1\leq j^{*}<j\leq \ell-1$, we would have $w_{j^*}< w_{j+1} < w_{j}$. According to the definition of up-down words, $w_{j}$ must be a top element and $w_{j+1}$ must be a bottom element in $w$.
If $w_{j^*}$ is a bottom element, then there is $w_{j^*}< w_{j+1}$ and the bottom element $w_{j^*}$ is to the left of the bottom element $w_{j+1}$. If $w_{j^*}$ is a top element, then there is $w_{j^*+1}<w_{j^*}< w_{j+1}$, and the bottom element $w_{j^*+1}$ is to the left of the bottom element $w_{j+1}$.

\item[(b)] Combining Lemma~\ref{132} and $(a)$, we get the desired result.
\end{itemize}

\qed

The enumeration of $3\underline{12}$-avoiding up-down words is similar to that of $1\underline{32}$-avoiding up-down words, and we omit a proof of the following theorem leaving it to the interested Reader.

\begin{thm}
The following two statements hold:
\begin{itemize}
\item[(a)] An up-down word $w$ is $3\underline{12}$-avoiding if and only if the top elements of $w$ are weakly increasing from left to right, i.e.,
$$t_1\leq t_2\leq\dots \leq t_{\lfloor\frac{\ell}{2}\rfloor}.$$
\item[(b)] An up-down word $w$ is $3\underline{12}$-avoiding if and only if $w$ is $\underline{312}$-avoiding. Thus, for  all $k\geq 2$ and $\ell \geq 0$, we have
$$N_{k,\ell}^{3\underline{12}}=N_{k,\ell}^{\underline{312}}.$$
\end{itemize}
\end{thm}

\subsection{\texorpdfstring{$2\underline{31}$}{Lg}-avoiding or \texorpdfstring{$2\underline{13}$}{Lg}-avoiding up-down words}
Our proof of the following lemma is very similar to the proof of Theorem~\ref{thm-132-av} (a), and thus is omitted.

\begin{lem}\label{231**}
In a $2\underline{31}$-avoiding up-down word, the bottom elements are weakly increasing from left to right, i.e.,
$$b_1\leq b_2\leq\dots \leq b_{\lceil\frac{\ell}{2}\rceil}.$$
\end{lem}

Note that unlike Theorem~\ref{thm-132-av} (a), we do not have ``if and only if'' statement in Lemma~\ref{231**} as demonstrated, e.g., by the word $12131$.


The following theorem shows that avoidance of the pattern $2\underline{31}$ is equivalent to avoidance of the classical pattern $231$ studied in~\cite{GKZ}.

\begin{thm}\label{2-31-is-eq-231}
An up-down word $w=w_1w_2\cdots w_{\ell}$
 is $2\underline{31}$-avoiding if and only if $w$ is $231$-avoiding.
\end{thm}

\begin{proof}
If $w$ has an occurrence of the pattern $2\underline{31}$ then it  clearly has an occurrence of the pattern 231. Thus, we just need to show that if $w$
 is $2\underline{31}$-avoiding, then $w$ is $231$-avoiding. Suppose that $w$ is $2\underline{31}$-avoiding, but there is an occurrence $w_{j_1}w_{j_2}w_{j_3}$ of the pattern $231$ in $w$, that is, $j_1<j_2<j_3$  and $w_{j_3}<w_{j_1}<w_{j_2}$. Among all such occurrences, we can pick one which has $j_3-j_1$ minimum possible.
\begin{itemize}

\item[(a)] If $w_{j_2}$ is a bottom element, then $w_{j_3}$ must be a top element by Lemma~\ref{231**}. Since $w_{j_3-1}<w_{j_3}$, we have that $w_{j_2}\neq w_{j_3-1}$. But then, $w_{j_2}$ and $w_{j_3-1}$ are bottom elements such that $w_{j_3-1}$ is to the right of $w_{j_2}$ and $w_{j_3-1}<w_{j_2}$ contradicting Lemma~\ref{231**}.

\item[(b)]If $w_{j_2}$ is a top element, we have the following cases to consider. If $j_3=j_2+1$, then $w_{j_1}w_{j_2}w_{j_3}$ is an occurrence of the pattern $2\underline{31}$, which is impossible. If $j_3\geq j_2+2$ and $w_{j_3}$ is a bottom element, according to the definition of up-down words and Lemma~\ref{231**}, we have $w_{j_2+1}\leq w_{j_3}$ and thus $w_{j_1}w_{j_2}w_{j_2+1}$ is an occurrence of the pattern $2\underline{31}$; contradiction. Finally, if $j_3\geq j_2+2$ and $w_{j_3}$ is a top element, then $w_{j_1}w_{j_2}w_{j_3-1}$ is an occurrence of the pattern $231$ with $w_{j_1}$ and $w_{j_3-1}$ being closer to each other than $w_{j_1}$ and $w_{j_3}$ contradicting our choice of $w_{j_1}w_{j_2}w_{j_3}$.
\end{itemize}
The proof is completed.
\end{proof}

The following statement is a direct corollary to Theorem~\ref{2-31-is-eq-231}.

\begin{cor} \label{cor 2.5}For all $k\geq 2$ and $\ell \geq 0$, we have
$$N_{k,\ell}^{2\underline{31}}=N_{k,\ell}^{231},$$ which is enumerated in Theorem~\ref{thm-from-prev-paper}.
\end{cor}

The enumeration of $2\underline{13}$-avoiding up-down words is similar to that of $2\underline{31}$-avoiding up-down words. Here we list all the results about the former objects, omitting the proofs.

\begin{thm}\label{thm 2.6}
The following two statements hold:
\begin{itemize}
\item[(a)] In an up-down $2\underline{13}$-avoiding word $w$, the top elements  are weakly increasing from left to right, i.e.,
$$t_1\geq t_2\geq\dots \geq t_{\lfloor\frac{\ell}{2}\rfloor}.$$
\item[(b)] An up-down word $w$ is $2\underline{13}$-avoiding if and only if $w$ is $213$-avoiding. Thus, for all $k\geq 2$ and $\ell \geq 0$, we have
$$N_{k,\ell}^{2\underline{13}}=N_{k,\ell}^{213},$$ which is enumerated in Theorem~\ref{thm-from-prev-paper}.
\end{itemize}
\end{thm}

Note that in Theorem~\ref{thm 2.6} (a) we do not have an ``if and only if'' statement, as shown by, e.g., the word 2313.

\subsection{\texorpdfstring{$1\underline{23}$}{Lg}-avoiding or \texorpdfstring{$3\underline{21}$}{Lg}-avoiding up-down words}
A description of $1\underline{23}$-avoiding up-down words is as follows.
\begin{lem}\label{123**}
An up-down word $w=w_1w_2\cdots w_{\ell}$
 is $1\underline{23}$-avoiding if and only if
$$b_1\geq b_2\geq\dots \geq b_{\lfloor\frac{\ell}{2}\rfloor}.$$
\end{lem}

\begin{proof}
For any $1\underline{23}$-avoiding up-down word $w$, if there exists $1\leq j\leq \lfloor\frac{\ell}{2}\rfloor-1$ such that $b_{j} < b_{j+1}$, then $b_jb_{j+1}t_{j+1}$ is an occurrence of the pattern $1\underline{23}$, which is a contradiction.

Conversely, if there is an occurrence $w_{j^{*}}w_{j}w_{j+1}$ of the pattern $1\underline{23}$  in $w$, where $1\leq j^{*}<j< \ell$, we would have $w_{j^*}< w_{j} < w_{j+1}$. According to the definition of up-down words, $w_{j}$ must be a bottom element, and $w_{j+1}$ must be a top element in $w$. If $w_{j^*}$ is a bottom element, then $w_{j^*}$ is to the left of $w_{j}$ and $w_{j^*}< w_{j}$. If $w_{j^*}$ is a top element, then the bottom element  $w_{j^*-1}\neq w_{j}$ is to the left of $w_{j}$ and $w_{j^*-1}< w_{j}$.

This completes the proof.
\end{proof}

We can now obtain the following enumerative result.

\begin{thm}\label{thm 2.8}
The following two statements hold, where $N_{k,\ell}^{\underline{132}}$ is enumerated in Section~\ref{sub-cons-132}:
\begin{itemize}
\item[(a)]For all $k\geq 2$ and $i \geq 0$, we have
$$S_{k,2i}^{1\underline{23}}=S_{k,2i}^{\underline{132}}.$$
\item[(b)]For all $k\geq 2$ and $i \geq 1$, we have
$$N_{k,2i+1}^{1\underline{23}}=N_{k,2i+1}^{\underline{132}}+\sum_{j=1}^{k-1}\binom{k-j}{2} N_{k-j+1,2i-2}^{\underline{132}}.$$
\end{itemize}
\end{thm}

\begin{proof}
(a) follows immediately from Lemmas~\ref{132} and \ref{123**}.

For (b), there are two cases to consider:
\begin{itemize}
\item $b_i\geq b_{i+1}$. These words are counted by $N_{k,2i+1}^{\underline{132}}$.
\item $b_i<b_{i+1}$. Then, $b_i$ is the minimum element in $w$. Suppose that $b_i=j$, where $1\leq j\leq k-1$. Then the word $w$ must be of the form $w' j w''$, where $w'$ is a $1\underline{23}$-avoiding up-down word of length $2i-2$ over $\{j, j+1, \dots, k\}$, and $w''$ is a down-up word of length $2$ over $\{j+1, \dots, k \}$. Thus, the words in question are counted by $\sum_{j=1}^{k-1}\binom{k-j}{2} N_{k-j+1,2i-2}^{\underline{132}}$.
\end{itemize}
This completes the proof.
\end{proof}

The case of enumeration of $3\underline{21}$-avoiding up-down words is similar to that of $1\underline{23}$-avoiding up-down words conducted above. Thus, we omit our proof of the following theorem.

\begin{thm}\label{thm 2.9}
The following three statements hold, where $N^{\underline{312}}_{k,\ell}$ is enumerated in Section~\ref{sec-3-2}:
\begin{itemize}
\item[(a)] An up-down word $w$
 is $3\underline{21}$-avoiding if and only if
$$t_1\leq t_2\leq\dots \leq t_{\lfloor\frac{\ell-1}{2}\rfloor}.$$
\item[(b)]For all $k\geq 2$ and $i \geq 0$, we have
$$N_{k,2i+1}^{3\underline{21}}=N_{k,2i}^{\underline{312}}.$$
\item[(c)]For all $k\geq 2$ and $i \geq 2$, we have
$$N_{k,2i}^{3\underline{21}}=N_{k,2i}^{\underline{312}}+
\sum_{j=2}^{k}\binom{j-1}{2} \left( N_{j,2i-3}^{\underline{312}}- \delta_{i,2}\right).$$
\end{itemize}
\end{thm}

\subsection{The remaining cases}
The remaining enumeration cases for vincular pattern-avoiding up-down words are obtained by applying the reverse and complement operations to our obtained results. We record these cases in the following two theorems.

\begin{thm}
For all $k\geq 2$ and $i \geq 0$, we have
$$N_{k,2i}^{\underline{12}3}=N_{k,2i}^{1\underline{23}},~~
N_{k,2i}^{\underline{21}3}=N_{k,2i}^{1\underline{32}},~~
N_{k,2i}^{\underline{13}2}=N_{k,2i}^{2\underline{13}}$$
and
$$N_{k,2i}^{\underline{31}2}=N_{k,2i}^{2\underline{31}},~~
N_{k,2i}^{\underline{23}1}=N_{k,2i}^{3\underline{12}},~~
N_{k,2i}^{\underline{32}1}=N_{k,2i}^{3\underline{21}}.$$
\end{thm}

\begin{thm}
For all $k\geq 2$ and $i \geq 0$, we have
$$N_{k,2i+1}^{\underline{12}3}=N_{k,2i+1}^{3\underline{21}},~~
N_{k,2i+1}^{\underline{21}3}=N_{k,2i+1}^{3\underline{12}},~~
N_{k,2i+1}^{\underline{13}2}=N_{k,2i+1}^{2\underline{31}}$$
and
$$N_{k,2i+1}^{\underline{31}2}=N_{k,2i+1}^{2\underline{13}},~~
N_{k,2i+1}^{\underline{23}1}=N_{k,2i+1}^{1\underline{32}},~~
N_{k,2i+1}^{\underline{32}1}=N_{k,2i+1}^{1\underline{23}}.$$
\end{thm}

\section{Concluding remarks}

In this paper, we not only enumerated all cases of length 3 vincular pattern-avoidance on alternating words providing a link, e.g., to the Stirling numbers of the second kind, but also discussed the structure of 123-avoiding up-down words of even length. As the result, we provided an alternative, combinatorial proof of the fact that these words are counted by the Narayana numbers. However, our combinatorial proof uses a bijection between Dyck paths and certain equivalence classes on words in question, along with a known relation on Narayana numbers. It is still desirable to solve the following problem.

\begin{prob} Provide a direct combinatorial proof of the fact that $123$-avoiding up-down words of even length are counted by the Narayana numbers, namely, find a bijection sending these words to Dyck paths. \end{prob}

Also, it would be interesting to describe the structure of 132-avoiding up-down words of even length, e.g., via the notion of a cut-pair introduced in this paper, and possibly provide an alternative proof of the fact that these words are counted by the Narayana numbers, as was shown in \cite{GKZ}. We leave this as an open research direction.

Finally, there are many other types of patterns studied in the literature (see Chapter~1 in \cite{Kitaev2011Patterns}) and one could study occurrences of these patterns on alternating words, which should bring more links to known combinatorial structures.

\section*{\bf Acknowledgments}
The authors are grateful to an anonymous referee for reading carefully the paper and providing many useful suggestions that improved the presentation. The work of the first and the third authors was supported by the 973 Project, the PCSIRT Project of the Ministry of Education and the National Science Foundation of China. The second author is grateful to the administration of the Center for Combinatorics at Nankai University for their hospitality during the author's stay in June--July 2015.

%

\end{document}